%% file: main.tex
\documentclass[10pt]{article} %
\usepackage[preprint]{tmlr}

\usepackage[T1]{fontenc}    %
\usepackage{microtype}      %
\usepackage[normalem]{ulem}

\usepackage{xcolor}

\usepackage{mathrsfs}
\usepackage{mathtools,amssymb}   %
\usepackage{bm}             %

\usepackage{booktabs}

\usepackage{my_notation}

\usepackage{amsthm}
\usepackage{thmtools}                   %

\usepackage[boxed]{algorithm2e}         %

\usepackage{csquotes}
\usepackage[british]{babel}     %

\usepackage[hidelinks]{hyperref}       %
\usepackage{bookmark}           %
\usepackage[capitalise]{cleveref}

\usepackage{my_theorems}

\usepackage{enumitem}           %

\title{A dimension-free Bernstein-type inequality\\for self-normalised martingales}

\author{\name Arya Akhavan \email arya.akhavan@stats.ox.ac.uk \\
      \addr Department of Statistics\\University of Oxford
      \AND
      \name Amitis Shidani \email amitis.shidani@stats.ox.ac.uk \\
      \addr Department of Statistics\\University of Oxford
      \AND
      \name Alex Ayoub \email aayoub@ualberta.ca\\
      \addr Department of Computing Science\\ University of Alberta
      \AND
      \name David Janz \email david.janz@stats.ox.ac.uk\\
      \addr Department of Statistics\\University of Oxford
      }

\def\month{MM}  %
\def\year{YYYY} %
\def\openreview{\url{https://openreview.net/forum?id=XXXX}} %

\begin{document}

\maketitle

\begin{abstract}
We introduce a dimension-free Bernstein-type tail inequality for self-normalised martingales, where the normalisation uses the predictable quadratic variation and the radius depends on the information gain of the observed covariance. As applications, we provide ellipsoidal confidence sequences for logistic regression with adaptively chosen Hilbert-valued covariates, and give instance-adaptive regret bounds for Hilbert-armed logistic bandits.
\end{abstract}

\input{sections/paper.tex}

\bibliography{references}
\bibliographystyle{tmlr}

\appendix
\clearpage
\input{sections/appendix.tex}
\end{document}

%% file: sections/paper.tex
\section{Introduction}

Let $(\Omega, \cF, \bP)$ be a probability space equipped with a filtration $(\cF_n)_{n \in \N}$. Let $\cH$ be a separable Hilbert space with inner product $\langle \cdot, \cdot \rangle$, norm $\|\cdot\|$, closed unit ball $B$ and identity operator $I$.

Consider a predictable $B$-valued sequence $(X_n)_{n \in \Np}$ and adapted real-valued sequence $(Y_n)_{n \in \Np}$ satisfying
\[
  \E [Y_n \mid \cF_{n-1}] = 0 \spaced{and} |Y_n| \leq 1 \quad \text{a.s. for all $n \in \Np$}\,.
\]
Define the martingale sequence
\[
  M_n = \sum_{j=1}^n Y_j X_j\,, \quad n \in \Np\,,
\]
along with its predictable and worst-case quadratic variation processes, defined respectively as
\[
  Q_n = \sum_{j=1}^n \E[Y_j^2 \mid \cF_{j-1}] (X_j \otimes X_j) \spaced{and} V_n = \sum_{j=1}^n X_j \otimes X_j\,.
\]
In this paper, we provide high-probability time-uniform control of the self-normalised process
\[
  Z_n(\rho) = \norm{(Q_n + \rho I)^{-\frac{1}{2}} M_n}\,,
\]
where $\rho \geq 1$ is a regularisation parameter, and demonstrate how this bound fits into the analyses of logistic regression with adaptively chosen covariates and logistic bandits. Since  $Z_n(\rho)$ is normalised by the predictable variance $Q_n$ and not the worst-case upper bound $V_n$, we refer to our result as Bernstein-type.

\section{Concentration result}
Our upper bound will be stated in terms of information gain \citep{seeger2008information}. For a trace class operator $A$ on $\cH$, define the information gain functional
\[
   \gamma(A) = \frac{1}{2} \log\det(I+ A)\,,
\]
where $\det$ denotes the Fredholm determinant. Our main concentration result is then as follows.

\begin{theorem}[Fixed $\rho$ result]\label{thm:bernstein-fixed}
    Fix $\rho\ge 1$. For every $y>0$,
    \[
        \P[\bigg]{\exists n \in \Np \colon
        Z_n(\rho)
        >
        \sqrt{\rho}
        +
        \frac{78\gamma(\rho^{-1} V_n)+4y}{\sqrt{\rho}}
        +
        \sqrt{18\gamma(\rho^{-1}V_n)+6y}} \leq 2e^{-y}\,.
    \]
\end{theorem}

Observe that setting $\rho \approx \gamma(\rho^{-1} V_n)$ would yield a right-hand side scaling with $\sqrt{\gamma(\rho^{-1} V_n)}$. We can achieve this while retaining the time-uniform nature of the bound at an $\iota_n := 1+ \log \log(en)$ cost. For this, define
\[
  \gamma^\star_n = \inf \{ \rho \geq 1 \colon \rho \geq \gamma(\rho^{-1}V_n)\} \spaced{and let} Z^\star_n := Z_n(\gamma^\star_n)\,.
\]
In particular, we have the following corollary controlling $Z^\star_n$.

\begin{corollary}[Adaptive result]\label{thm:bernstein}
    For all $y > 0$,
    \[
      \P[\bigg]{\exists n \in \Np \colon
      Z^\star_n
      >
      100\roundb[\Big]{\sqrt{\gamma^\star_n + y + \iota_n}
      + \frac{y + \iota_n}{\sqrt{\gamma^\star_n}}}} \leq e^{-y}\,.
    \]
\end{corollary}
The corollary follows by a union bound over a geometric grid of regularisers. Constants here and in the fixed-$\rho$ result can be made significantly tighter at the cost of more cumbersome statements and proofs.

\subsection{Proof overview} Our proof is based on adaptively splitting the martingale into small and large directions.

\begin{definition}
Let $M=(M_n)_{n\in\N}$ be an $\cH$-valued martingale. An \emph{admissible predictable support} for $M$ is a nested sequence
$(\cU_n)_{n\in\N}$ of finite-dimensional random subspaces of $\cH$ such that $\cU_0=\{0\}$,
the orthogonal projection onto $\cU_n$ is $\cF_{n-1}$-measurable for each $n\in\Np$, and
\[
  \cU_{n-1}\subseteq \cU_n
  \spaced{and}
  M_n\in\cU_n
  \quad \text{a.s. for all $n\in\Np$}.
\]
\end{definition}

The next result is an adaptation of the Bernstein-type inequality for self-normalised processes on $\Rd$ of \citet{faury2020improved} to martingales supported on a predictable sequence of finite-dimensional subspaces.

\begin{theorem}[name={},restate=headBoundThm]\label{thm:bernstein-predictable-dim}
    Suppose that $M$ admits an admissible predictable support $(\cU_n)_{n\in\N}$, and write
    $D_n=\dim(\cU_n)$. Then, for any $y>0$ and $\rho \ge 1$,
    \[
        \P[\bigg]{\exists n \in \Np \colon
        \norm{(Q_n + \rho I)^{-\frac{1}{2}}M_n} >
        \frac{1}{2}\sqrt{\rho}
        +
        \frac{2}{\sqrt{\rho}}\roundb[\big]{\gamma(\rho^{-1}Q_n)+2D_n+y}} \leq e^{-y}\,.
    \]
\end{theorem}
The proof of \cref{thm:bernstein-predictable-dim} extends the finite-dimensional result of \citet{faury2020improved} through a simple stopping time argument.

\begin{theorem}[\citet{abbasi2013online}, Corollary 3.5]\label{thm:hoeffding-fixed}
    For all $y>0$ and $\rho > 0$,
    \[
      \P{\exists n \in \Np \colon \norm{(V_n + \rho I)^{-\frac{1}{2}}M_n} > \sqrt{2(\gamma(\rho^{-1}V_n) + y)}} \leq e^{-y}\,.
    \]
\end{theorem}

The proof of the main result then splits the martingale according to an online spectral truncation of the covariates.
Fix $\rho\ge1$. We maintain a predictable nested sequence of finite-dimensional subspaces
$(\cH_n)_{n\in\N}$, with orthogonal projections $(\Pi_n)_{n\in\N}$.
Starting from $\cH_0=\{0\}$, when $X_n$ is revealed we decompose
\[
  X_n^+ = \Pi_{n-1}X_n
  \spaced{and}
  X_n^- = \Pi_{n-1}^{\perp}X_n.
\]
We then enlarge $\cH_{n-1}$ by adding the eigendirections of the residual covariance
$V_n^-=\sum_{j=1}^n X_j^-\otimes X_j^-$ whose eigenvalues have reached the scale $\rho$. The construction gives, for all $n$,
\[
  \norm{V_n^-}_{\op}\le 2\rho,\qquad
  \dim(\cH_{n-1})\le 3\,\gamma(\rho^{-1}V_n^-),
\]
Correspondingly, write
\[
  M_n^+ = \sum_{j=1}^n Y_jX_j^+,\qquad
  M_n^- = \sum_{j=1}^n Y_jX_j^-,
\]
and let
\[
  Q_n^+ =
  \sum_{j=1}^n \E[Y_j^2\mid\cF_{j-1}]\,X_j^+\otimes X_j^+,
  \qquad
  Q_n^- =
  \sum_{j=1}^n \E[Y_j^2\mid\cF_{j-1}]\,X_j^-\otimes X_j^- .
\]
Now, by the triangle inequality,
\[
  \norm{(Q_n + \rho I)^{-\frac{1}{2}}M_n} \leq \norm{(Q_n + \rho I)^{-\frac{1}{2}}M_n^+} + \norm{(Q_n + \rho I)^{-\frac{1}{2}}M_n^-}\,.
\]

An elementary argument using the cutting point $\rho$ to compare $Q_n$ and $Q_n^+$ in positive-semidefinite order yields that
\[
  \norm{(Q_n+\rho I)^{-\frac12}M_n^+}
  \le
  2\norm{(Q_n^+ + \rho I)^{-\frac12}M_n^+}.
\]

For the tail, we use a different consequence of the same cut.
Since $Q_n+\rho I\succeq \rho I$ and $\norm{V_n^-}_{\op}\le2\rho$,
\[
  \norm{(Q_n+\rho I)^{-\frac12}M_n^-}
  \le
  \frac{\norm{M_n^-}}{\sqrt{\rho}}
  \le
  \sqrt{\frac{\norm{V_n^-+\rho I}_{\op}}{\rho}}\,
  \norm{(V_n^-+\rho I)^{-\frac12}M_n^-}
  \le
  \sqrt{3}\norm{(V_n^-+\rho I)^{-\frac12}M_n^-}.
\]
Applying the predictable support Bernstein inequality (\cref{thm:bernstein-predictable-dim}) to the head and using $\dim(\cH_{n-1})\le 3\,\gamma(\rho^{-1}V_n^-)$, and the Hoeffding inequality (\cref{thm:hoeffding-fixed}) to the tail and taking the union bound, we obtain that with probability at least $1-2e^{-y}$, for all $n \geq 1$,
\[
  \norm{(Q_n+\rho I)^{-\frac12}M_n}
  \lesssim
  \sqrt{\rho}
  +
  \frac{\gamma(\rho^{-1}Q_n^+)+\gamma(\rho^{-1}V_n^-)+y}{\sqrt{\rho}}
  +
  \sqrt{\gamma(\rho^{-1}V_n^-)+y}\,.
\]
From here, bounding $\gamma(\rho^{-1}Q_n^+)$ and $\gamma(\rho^{-1}V_n^-)$ in terms of $\gamma(\rho^{-1}V_n)$ completes the proof.

\subsection{Comparison of concentration result to concurrent results}

Whereas all previously published self-normalised Bernstein-type bounds feature an explicit dependence on the dimension of the Hilbert space \citep{faury2020improved,lee2024unified,ziemann2024vector}, our result depends on the information gain. This dependence may be much smaller than the ambient dimension.

Concurrently to this work, \citet{chugg2025variational,metelli2025generalized} also obtained self-normalised Bernstein-type inequalities. The result of \citet{chugg2025variational} is based on PAC-Bayes methods and is finite-dimensional. The result of \citet{metelli2025generalized} takes a different route, reducing the analysis to one-dimensional Sherman--Morrison updates for the process $Z_{n-1}(\rho)\mapsto Z_n(\rho)$. Both bounds are sharper than \cref{thm:bernstein-fixed} in their variance dependence: their concentration radii involve $\gamma(\rho^{-1}Q_n)$, whereas \cref{thm:bernstein-fixed} is stated in terms of $\gamma(\rho^{-1}V_n)$. The passage from $Q_n$ to $V_n$ in \cref{thm:bernstein-fixed} occurs at the final simplification of the truncation argument. Before this simplification, the proof gives a bound involving
\[
  \gamma(\rho^{-1}Q_n^+) \spaced{and} \gamma(\rho^{-1}V_n^-),
\]
where the second term comes from the residual directions left after the spectral cut. Thus the worst-case covariance appears only through the small-direction part of the decomposition. In problems where the conditional variance is concentrated on a smaller effective subspace, this intermediate form can be substantially sharper than the displayed bound in \cref{thm:bernstein-fixed} suggests.

The applications below use the concentration result only through the resulting self-normalised confidence radius. This makes it possible to substitute sharper compatible Bernstein inequalities without changing the learning-theoretic reductions. We make this explicit in the logistic bandit application.

\subsection{Related work}

Self-normalised martingale inequalities have a long history. Their study was partly motivated by Studentised statistics beyond the Gaussian setting \citep{efron1969student,logan1973limit}, and was later developed into a general theory by \citet{de1999general,de2004self,pena2009self}. The central mechanism is the method of mixtures: suitable exponential supermartingales are integrated against a mixing distribution to produce time-uniform bounds for ratios of the form $A_n/\sqrt{B_n+\rho}$. In sequential least squares, \citet{abbasi2011improved} used this method to obtain a vector-valued self-normalised inequality with subgaussian increments. The same argument extends to reproducing kernel Hilbert spaces, as shown in the thesis of \citet{abbasi2013online}. The Hoeffding-type inequality in \cref{thm:hoeffding-fixed} is an instance of this result. Our proof uses the finite-dimensional Bernstein inequality of \citet{faury2020improved} as an input; the infinite-dimensional step is based on an online spectral truncation. This truncation is closely related to the effective-dimension viewpoint in kernel learning. \citet{zhang2002effective,zhang2005learning} introduced effective data dimension to measure the number of statistically relevant kernel directions at a given regularisation scale; the information-gain functional $\gamma(\rho^{-1}V_n)$ is closely related to that effective dimension. A related but distinct line of work studies dimension-free concentration for norms of vector-valued martingales. Martingale methods for sums of Banach-valued random variables appear in \citet{ledoux1991probability}, and \citet{pinelis1994optimum} proved sharp Hoeffding-, Bennett- and Bernstein-type inequalities for martingales in $2$-smooth Banach spaces, including Hilbert spaces. More recent results refine this direction, for example through scalar-sharp Azuma--Hoeffding constants for Hilbert-valued martingales \citep{luo2022azuma} and empirical Bernstein inequalities for vector-valued martingales \citep{martinez2024empirical}. These inequalities control $\norm{M_n}$ directly, and therefore do not retain the covariance geometry of the covariates.

\section{Applications to logistic regression and bandits}\label{sec:applications}

We consider online logistic regression and bandits in an infinite-dimensional Hilbert space. Let
\[
\mu(u) = 1/(1+e^{-u})\,, \qquad \rV(u) = \mu(u)(1-\mu(u))
\]
denote the logistic mean and variance functions. We suppose a predictable sequence of covariates $(X_n)_{n \in \Np}$ in a compact set $\cX \subset B$, and adapted responses $(Y_n)_{n \in \Np}$ in $[0,1]$, having conditional means
\[
  \E[Y_n \mid \cF_{n-1}] = \mu(\langle f^\star, X_n\rangle) \quad \text{a.s. for a fixed $f^\star \in \cH$}\,.
\]
We assume knowledge of a bound $b > 0$ such that $\norm{f^\star} \leq b$. We do not require $(Y_n)_{n \in \Np}$ to be Bernoulli.

For the theoretical analysis, we define the residual process and its quadratic variation
\[
  M_n = \sum_{j=1}^n \epsilon_j X_j
  \spaced{where}
  \epsilon_j = Y_j-\mu(\langle f^\star,X_j\rangle),
  \qquad
  Q_n = \sum_{j=1}^n \E[\epsilon_j^2\mid\cF_{j-1}]X_j\otimes X_j\,.
\]
In terms of the above residual process, we assume the existence of a concentration radius, as follows.
\begin{assumption}[Concentration radius] In what follows, let $\beta_n(\rho,y)$ be any radius such that the event
\[
  \cC = \bigcap_{n \geq 1}\curlyb[\big]{\norm{(Q_n+\rho I)^{-\frac12}M_n}\le \beta_n(\rho,y)} \spaced{satisfies} \P[\big]{\cC}
  \ge 1-Ce^{-y}\,,
\]
and $n \mapsto \beta_n(\rho, y)$ is nondecreasing (if the latter is not satisfied, replace $\beta_n(\rho,y)$ with its running maximum).
\end{assumption}
For example, one may take $\beta_n(\rho,y)$ to be the upper bound from \cref{thm:bernstein-fixed}, with $C=2$.

Finally, define the logistic information process at $f \in \cH$ by
\[
  A_n(f) = \sum_{j=1}^n \rV(\langle f, X_j\rangle) (X_j \otimes X_j)\,, \spaced{and observe that} Q_n \preceq A_n(f^\star)\,,
\]
with equality when the responses are Bernoulli (where $\preceq$ denotes the usual order on positive semidefinite operators). This follows since Bernoulli random variables have the maximal variance on $[0,1]$.

\subsection{Anytime confidence sequence for logistic regression under adaptive design}

Define the logistic loss function $\ell \colon \R \times [0,1] \to \R$ by
\[
    \ell(u, y) = - y \log(\mu(u)) - (1-y) \log(1 - \mu(u))\,.
\]
For $\rho > 0$, define the cumulative $\rho$-regularised loss and its minimiser the logistic ridge regressor
\[
    \cL_n(f) = \sum_{i=1}^n \ell(\langle f, X_i \rangle, Y_i) + \frac{\rho}{2} \norm{f}^2 \spaced{and} \hat f_n \in \argmin_{f \in \cH} \cL_n(f)\,.
\]
Let $\nabla^2\cL_n(f)$ denote the Hessian operator of $\cL_n$ at $f$, noting that
\[
  \nabla^2 \cL_n(f) = A_n(f) + \rho I\,, \spaced{and define} A^\star_n = A_n(f^\star)\,, \qquad \hat A_n = A_n(\hat f_n)\,.
\]
Our confidence sequence is now given in terms of the above quantities and a concentration radius $\beta_n(\rho, y)$.

\begin{theorem}[Logistic confidence sequences]\label{thm:confidence-sequence}
    Fix $y>0$ and $\rho>0$. Define
    \[
      u_n = \frac{\beta_n(\rho,y)}{\sqrt{\rho}}+b\,, \qquad \omega_n= u_n(6+2u_n^3)\sqrt{\rho}\,, \qquad \alpha_n = u_n(1+u_n)\,.
    \]
    Then, on the event $\cC$, for all $n \geq 1$, we have the inequalities
    \[
      \norm{(A^\star_n + \rho I)^{\frac{1}{2}}(f^\star - \hat f_n)}
        \vee
        \norm{(\hat{A}_n + \rho I)^{\frac{1}{2}}(f^\star - \hat f_n)}
        \leq \omega_n \spaced{and} \norm{f^\star - \hat f_n} \leq \alpha_n\,.
    \]
\end{theorem}

The above inequalities give confidence ellipsoids for $f^\star$ of the form
\[
  \{f \in \cH \colon \norm{(B_n+\rho I)^{\frac{1}{2}}(f - \hat f_n)} \leq \omega_n(\rho,y) \} \spaced{with} B_n \in \{ A^\star_n,  \hat{A}_n\}\,.
\]
If the responses $Y$ are Bernoulli, the resulting confidence ellipsoids are second-order, scaling with either the true variances $A^\star_n$ or the estimates $\hat A_n$. For general responses in $[0,1]$, since $\rV(u) \sim \mu(u)$ as $u \to -\infty$ and $\rV(u) \sim (1-\mu(u))$ as $u \to +\infty$, the confidence sequences scale with the first-order (mean) information.

While our construction is based on the general approach of \citet{faury2020improved}, our confidence sequences are valid uniformly over time, and not just after sufficient exploration in all directions. This is achieved through the use of a lower bound from \citet{sun2019generalized}, previously leveraged by \citet{lee2024unified} for finite-dimensional generalised linear models. Our dependence on $b$ is worse than that of \citet{lee2024unified}; however, the technique of \citet{lee2024unified} is based on uniform distributions on ellipsoids, and thus does not transfer to infinite dimensions; moreover the result of \citet{lee2024unified} requires the responses to be genuinely Bernoulli, and not just bounded in $[0,1]$, making the result inapplicable in, for example, the standard setting of reinforcement learning with bounded rewards \citet{bakhtiari2026eluder}. The confidence sequence is not easily comparable to that of \citet{metelli2025generalized}, who use non-ellipsoidal sets; the relevant comparison is made through the induced bandit regret bound below.

\subsection{Logistic bandits in Hilbert spaces}

We now turn to logistic bandits, an online learning setting motivated by sequential decision-making tasks such as recommendation systems and clinical trials \citep{zelen1969play,li2010contextual,li2017provably}. Given a closed arm-set $\cX \subset B$, the aim is to select arms $X_j \in \cX$ sequentially to minimise cumulative (pseudo-)regret, defined
\[
  R_n = \sum_{j=1}^n \roundb[\big]{\mu(\langle f^\star, x^\star\rangle) - \mu(\langle f^\star, X_j\rangle)}
  \spaced{where} x^\star \in \argmax_{x \in \cX} \langle f^\star, x \rangle\,.
\]

A natural approach, given our previously constructed confidence sequences, is to employ the LinUCB algorithm \citep{abbasi2011improved}. Adopting the definitions of \cref{thm:confidence-sequence}, define the inflated radius
\[
  \bar{\omega}_n = \omega_n(1+\alpha_n)\,,
\]
and the inflated confidence sets
\[
  \bar{\cE}_0 = \curlyb{ f \in \cH \colon \norm{f} \leq b}\,, \qquad  \bar{\cE}_n = \curlyb{f \in \cH \colon \norm{(\hat{A}_n + (\rho + \omega^2_n) I)^{\frac{1}{2}}(f - \hat f_n)}
        \leq \bar{\omega}_n}\,, \qquad n \geq 1\,.
\]
Then, at round $n$, the algorithm chooses
\[
  X_n\in \argmax_{x\in\cX} \max_{f \in \bar{\cE}_{n-1}}\langle f,x\rangle\,.
\]
The upcoming regret bound depends on the variance of the optimal arm.
\[
  \nu_\star = \rV(\langle f^\star, x^\star \rangle)\,.
\]

\begin{theorem}[name={},restate=regretThm]\label{thm:raw_reg}
Fix $y>0$ and $\rho\ge1$. Suppose that the confidence sequence of \cref{thm:confidence-sequence}
holds with probability at least $1-Ce^{-y}$, and replace $\omega_n$ and $\bar\omega_n$ by their
running maxima if necessary. Then the regret of the optimistic rule above satisfies, with probability
at least $1-Ce^{-y}$, for all $n\in\Np$,
    \[
        R_n
        \le
        12\bar\omega_n\sqrt{\nu_\star n \gamma(\rho^{-1}A^\star_n)}
        +
        64e^{b}\bar\omega_n^2 \gamma(\rho^{-1}A^\star_n).
    \]
In particular, when $\beta_n(\rho,y)$ is chosen from \cref{thm:bernstein-fixed}, one may take $C=2$.
\end{theorem}

Like the bound of \citet{metelli2025generalized}, our result depends on $\nu_\star$. For Bernoulli rewards, this is a second-order bound: the regret is small whenever the optimal arm’s variance is low; for general bounded responses, since $\nu_\star = \mu(\langle f^\star, x^\star\rangle)(1-\mu(\langle f^\star, x^\star\rangle))$, this is a \emph{first-order} regret bound---it is small when the mean reward of the optimal arm $\mu(\langle f^\star, x^\star\rangle)$ is close to either boundary of the response interval $[0,1]$.

We now provide a more explicit version of our result and compare directly with \citet{metelli2025generalized}.

\begin{corollary}[Regret with the variance-weighted radius of \citet{metelli2025generalized}]
\label{cor:regret}
Assume in addition that the responses are Bernoulli, so that $Q_n=A^\star_n$.
Instantiate $\beta_n(\rho,y)$ in \cref{thm:confidence-sequence} with the variance-weighted
self-normalised Bernstein inequality of \citet{metelli2025generalized}. Then there exists a
universal constant $C>0$ such that, with probability at least $1-e^{-y}$, for all $n\in\Np$
satisfying
\[
  \rho
  \ge
  C\roundb[\big]{y+\iota_n}
  \roundb[\big]{1+\gamma(\rho^{-1}Q_n)},
\]
the optimistic rule of \cref{thm:raw_reg} satisfies
\[
  R_n
  \le
  C(1+b)^6\sqrt{\rho\,\nu_\star n\,\gamma(\rho^{-1}Q_n)}
  +
  C e^b(1+b)^{12}\rho\,\gamma(\rho^{-1}Q_n).
\]
\end{corollary}

\begin{remark}[Comparison with \citet{metelli2025generalized}]
  For comparison, \citet[Theorem~6.2]{metelli2025generalized}, specialised to the Bernoulli logistic
case, simplified and rewritten in our notation gives the following: there exists a universal constant $C>0$ such that, with
probability at least $1-e^{-y}$, for all $n\in\Np$ satisfying
\[
  \rho
  \ge
  C\roundb[\big]{y+\iota_n}
  \roundb[\big]{1+\gamma(\rho^{-1}V_n)},
\]
their GKB-UCB algorithm satisfies
\[
  R_n
  \le
  C(1+b)^2\sqrt{\rho\,\nu_\star n\,\gamma(\rho^{-1}V_n)}
  +
  C e^b(1+b)^4\rho\,\gamma(\rho^{-1}V_n).
\]
Thus, when \cref{cor:regret} is instantiated with the radius of \citet{metelli2025generalized}, the
two bounds differ in two ways. Our plug-in result replaces the worst-case gain
$\gamma(\rho^{-1}V_n)$ by the variance-weighted gain $\gamma(\rho^{-1}Q_n)$. On the other hand, our confidence-sequence route
pays heavier factors in $b$. This loss comes from us insisting on simple ellipsoidal confidence sets around the logistic ridge
estimator and then inflating the $\hat A_n$-ellipsoid to make optimism easy to state. The algorithm of
\citet{metelli2025generalized} uses a more complex confidence set and optimisation problem, which then avoids
some of this self-concordance-based $b$-factor inflation.
\end{remark}

The final term involving $e^b$ is standard: it represents the difficulty of ensuring adequate exploration across roughly $\gamma(\rho^{-1}Q_n)$ directions of the Hilbert space. See \citet{faury2020improved} for a detailed discussion.

A Hoeffding equivalent of \cref{thm:raw_reg} may be obtained by using \cref{thm:hoeffding-fixed} in place of our Bernstein bound, \cref{thm:bernstein-fixed}. The result is a bound of the same form as in \cref{thm:raw_reg}, but with $\nu_\star$ replaced by its worst-case value $1/4$. Such results were first established by \citet{srinivas2009gaussian}. Tighter regret bounds may be obtained by elimination-based algorithms \citep{valko2013finite} or large-batch algorithms \citep{li2022gaussian}, but it is a question of whether the improvement is due to better algorithms or tighter analysis.

%% file: sections/appendix.tex
\section{Adaptive truncation level (proof of \cref{thm:bernstein})}
\label{sec:adaptive-bounds}

We now derive the adaptive version of the fixed-$\rho$ inequality by a geometric-grid argument.
Recall that
\[
  \gamma^\star_n
  =
  \inf\{\rho\ge1:\rho\ge\gamma(\rho^{-1}V_n)\},
  \qquad
  Z^\star_n=Z_n(\gamma^\star_n),
  \qquad
  \iota_n=1+\log\log(en).
\]

\begin{proof}[Proof of \cref{thm:bernstein}]
Let
\[
  \eta=\frac{21}{20},\qquad
  \rho_h=\eta^{h-1},\qquad
  y_h=y+\log(4h^2),
  \qquad h\in\Np.
\]
For each $h\in\Np$, define the event
\[
  \cE_h =
  \curlyb[\bigg]{\forall n\in\Np,\
  Z_n(\rho_h)
  \le
  \sqrt{\rho_h}
  +
  \frac{78\gamma(\rho_h^{-1}V_n)+4y_h}{\sqrt{\rho_h}}
  +
  \sqrt{18\gamma(\rho_h^{-1}V_n)+6y_h}}\,.
\]
By \cref{thm:bernstein-fixed},
\[
  \P{\cE_h^\complement}\le 2e^{-y_h}
  =
  \frac{e^{-y}}{2h^2}.
\]
Hence, by a union bound,
\[
  \P[\bigg]{\bigcap_{h\ge1}\cE_h}
  \ge
  1-\frac{e^{-y}}{2}\sum_{h=1}^\infty\frac{1}{h^2}
  \ge
  1-e^{-y}.
\]

Work on $\bigcap_{h\ge1}\cE_h$ and fix $n\in\Np$. Write
\[
  r=\gamma^\star_n
  \spaced{and}
  a=y+\iota_n.
\]
Since $V_n$ is finite-rank, the map $\rho\mapsto\gamma(\rho^{-1}V_n)$ is continuous on $(0,\infty)$, so
$r$ is feasible in the definition of $\gamma^\star_n$:
\[
  \gamma(r^{-1}V_n)\le r.
\]
Choose $h\in\Np$ such that
\[
  \rho_h\le r<\eta\rho_h.
\]
Then $\rho_h^{-1}\le \eta r^{-1}$, and using $\gamma(cA)\le c\gamma(A)$ for $c\ge1$ and positive
finite-rank $A$,
\[
  \gamma(\rho_h^{-1}V_n)
  \le
  \eta\gamma(r^{-1}V_n)
  \le
  \eta r.
\]

We next control the grid index. Since $r$ is at most the feasible value $1\vee\gamma(V_n)$,
\[
  r\le 1\vee\gamma(V_n)\le n,
\]
where the final inequality follows from $\tr(V_n)\le n$ and $\log(1+x)\le x$. Since
$\log(21/20)>1/21$, the chosen index satisfies
\[
  h\le 1+21\log n.
\]
Therefore, with $\iota_n=1+\log\log(en)$,
\[
  \log(4h^2)
  \le
  \log\!\roundb[\big]{4(1+21\log n)^2}
  \le
  10\iota_n.
\]
Consequently,
\[
  y_h\le y+10\iota_n\le 10a.
\]

By monotonicity in the regularisation parameter and the event $\cE_h$,
\[
\begin{aligned}
  Z^\star_n
  =
  Z_n(r)
  \le
  Z_n(\rho_h)
  \le
  \sqrt{\rho_h}
  +
  \frac{78\gamma(\rho_h^{-1}V_n)+4y_h}{\sqrt{\rho_h}}
  +
  \sqrt{18\gamma(\rho_h^{-1}V_n)+6y_h}.
\end{aligned}
\]
Using $\rho_h\le r$, $\rho_h\ge r/\eta$, $\gamma(\rho_h^{-1}V_n)\le\eta r$ and $y_h\le10a$, we obtain
\[
\begin{aligned}
  Z^\star_n
  &\le
  \sqrt r
  +
  78\eta^{3/2}\sqrt r
  +
  40\sqrt{\eta}\frac{a}{\sqrt r}
  +
  \sqrt{18\eta r+60a}.
\end{aligned}
\]
Since $\eta=21/20$, the numerical bounds
\[
  78\eta^{3/2}\le 86,\qquad
  40\sqrt\eta\le 42,\qquad
  \sqrt{18\eta r+60a}\le 8\sqrt{r+a}
\]
give
\[
  Z^\star_n
  \le
  95\sqrt{r+a}
  +
  42\frac{a}{\sqrt r}
  \le
  100\roundb[\bigg]{\sqrt{r+a}+\frac{a}{\sqrt r}}.
\]
Substituting back $r=\gamma^\star_n$ and $a=y+\iota_n$ proves the claimed bound for the fixed $n$.
Since $n$ was arbitrary, the conclusion holds uniformly over all $n\in\Np$ on an event of probability at least
$1-e^{-y}$.
\end{proof}

\clearpage
\section{Bernstein-type inequality with predictable support (proof of \cref{thm:bernstein-predictable-dim})}
\label{appendix:predictable-dim}

We restate the theorem here for convenience.

\headBoundThm*

\begin{proof}
We use the finite-dimensional Bernstein inequality of \citet[Theorem~1]{faury2020improved} as a black box.
In our notation, their result implies the following. Let $(\tilde X_n)_{n\in\Np}$ be a predictable
$\R^d$-valued sequence with $\norm{\tilde X_n}\le1$, let $(Y_n)_{n\in\Np}$ be a real-valued martingale
difference sequence with $|Y_n|\le1$, and define
\[
  \tilde M_n=\sum_{j=1}^nY_j\tilde X_j
  \spaced{and}
  \tilde Q_n=\sum_{j=1}^n\E[Y_j^2\mid\cF_{j-1}]\tilde X_j\tilde X_j\tran .
\]
Then, for every $a>0$ and $\rho>0$,
\[
  \P[\bigg]{\exists n\in\Np\colon
  \norm{(\tilde Q_n+\rho I_d)^{-\frac12}\tilde M_n}
  >
  \frac{\sqrt{\rho}}{2}
  +
  \frac{2}{\sqrt{\rho}}\roundb[\big]{\gamma(\rho^{-1}\tilde Q_n)+d\log2+a}}
  \le e^{-a}.
\]
Indeed, \citet[Theorem~1]{faury2020improved} states the same event with the term
\[
  \frac{2}{\sqrt{\rho}}
  \log\!\Big(
  \frac{\det(\tilde Q_n+\rho I_d)^{1/2}\rho^{-d/2}}{\delta}
  \Big)
  +
  \frac{2d\log2}{\sqrt{\rho}},
\]
which is exactly the preceding display after setting $a=\log(1/\delta)$ and using
\[
  \frac12\log\det(I_d+\rho^{-1}\tilde Q_n)=\gamma(\rho^{-1}\tilde Q_n).
\]

We now reduce the predictable-support case to this finite-dimensional result. First note that, without changing
either $M$ or $Q$, we may replace $X_n$ by its projection onto $\cU_n$. Indeed, since
$M_{n-1}\in\cU_{n-1}\subseteq\cU_n$ and $M_n\in\cU_n$, we have $Y_nX_n=M_n-M_{n-1}\in\cU_n$ almost surely.
Thus
\[
  Y_n(I-\Pi_n)X_n=0\quad\text{a.s.},
\]
where $\Pi_n$ is the orthogonal projection onto $\cU_n$. Since $\Pi_n$ and $X_n$ are $\cF_{n-1}$-measurable,
\[
  \E[Y_n^2\mid\cF_{n-1}](I-\Pi_n)X_n=0,
\]
and hence the predictable quadratic variation is also unchanged by this replacement.

For a fixed integer $d\ge1$, define the stopping time
\[
  \tau_d=\inf\{n\in\Np:D_n>d\}.
\]
Equivalently, since $D_n$ is $\cF_{n-1}$-measurable and nondecreasing, the indicator
$\1{j<\tau_d}$ is predictable. Consider the stopped process
\[
  M^{(d)}_n=\sum_{j=1}^n \1{j<\tau_d}Y_jX_j
\]
and its predictable quadratic variation
\[
  Q^{(d)}_n=\sum_{j=1}^n \1{j<\tau_d}\E[Y_j^2\mid\cF_{j-1}]X_j\otimes X_j.
\]
Before time $\tau_d$, the process is supported on subspaces of dimension at most $d$.

Since $\cH$ is separable and the projections onto the nested subspaces $\cU_n$ are predictable, a predictable
Gram--Schmidt construction gives an orthonormal coordinate system that is enlarged only when the predictable
support grows. In the first $d$ coordinates, the stopped covariates form a predictable $\R^d$-valued sequence
$(\tilde X^{(d)}_j)_{j\in\Np}$ with $\norm{\tilde X^{(d)}_j}\le1$. The coordinate representation preserves
inner products, norms and quadratic variations, so the finite-dimensional Faury bound applies to
$M^{(d)}$ and $Q^{(d)}$.

Apply the displayed finite-dimensional bound with
\[
  a_d=y+\log(2d^2).
\]
We obtain
\[
  \P[\bigg]{\exists n\in\Np\colon D_n\le d\ \text{and}\
  \norm{(Q_n+\rho I)^{-\frac12}M_n}
  >
  \frac{\sqrt{\rho}}{2}
  +
  \frac{2}{\sqrt{\rho}}\roundb[\big]{\gamma(\rho^{-1}Q_n)+d\log2+y+\log(2d^2)}}
  \le \frac{e^{-y}}{2d^2}.
\]
Taking a union bound over $d\ge1$ gives failure probability at most $e^{-y}$. On the resulting event, for every $n$ with $D_n\ge1$, choosing $d=D_n$ yields
\[
  \norm{(Q_n+\rho I)^{-\frac12}M_n}
  \le
  \frac{\sqrt{\rho}}{2}
  +
  \frac{2}{\sqrt{\rho}}\roundb[\big]{\gamma(\rho^{-1}Q_n)+D_n\log2+\log(2D_n^2)+y}.
\]
Finally, the elementary bound
\[
  d\log2+\log(2d^2)\le 2d,\qquad d\in\Np,
\]
gives the claimed display. The case $D_n=0$ is trivial since then $M_n=0$.
\end{proof}

\clearpage
\section{Fixed-$\rho$ Bernstein concentration theorem (Proof of \cref{thm:bernstein-fixed})}
\label{app:proof-bernstein-fixed}

We complete the proof of \cref{thm:bernstein-fixed}. The main text already gives the
structure of the argument: split the martingale according to a predictable spectral cut,
apply \cref{thm:bernstein-predictable-dim} to the head and \cref{thm:hoeffding-fixed} to the tail,
and then compare the resulting information-gain terms to $\gamma(\rho^{-1}V_n)$.
Here we record this reduction with constants, and then prove the truncation and comparison estimates
used in the reduction.

Fix $\rho\ge1$ and write
\[
  \Gamma_n:=\gamma(\rho^{-1}V_n).
\]
Let $(\cH_n)_{n\in\N}$ be the predictable truncation constructed in \cref{lem:online-truncation}
below, and let $\Pi_n$ be the orthogonal projection onto $\cH_n$. Define
\[
  X_j^+ = \Pi_{j-1}X_j,
  \qquad
  X_j^- = \Pi_{j-1}^{\perp}X_j,
\]
and
\[
  M_n^+ = \sum_{j=1}^n Y_jX_j^+,\qquad
  M_n^- = \sum_{j=1}^n Y_jX_j^-.
\]
Also set
\[
  Q_n^+
  =
  \sum_{j=1}^n\E[Y_j^2\mid\cF_{j-1}]X_j^+\otimes X_j^+,
  \qquad
  Q_n^-
  =
  \sum_{j=1}^n\E[Y_j^2\mid\cF_{j-1}]X_j^-\otimes X_j^-,
\]
and
\[
  V_n^-=\sum_{j=1}^nX_j^-\otimes X_j^-.
\]

The deterministic comparison in \cref{lem:head-tail-comparison} gives, for every $n\in\Np$,
\[
  \norm{(Q_n+\rho I)^{-\frac12}M_n}
  \le
  2\norm{(Q_n^+ + \rho I)^{-\frac12}M_n^+}
  +
  \sqrt{3}\norm{(V_n^- + \rho I)^{-\frac12}M_n^-}.
\]
The head martingale $M^+$ has admissible predictable support $\cU_n=\cH_{n-1}$, hence by
\cref{lem:online-truncation},
\[
  D_n:=\dim(\cU_n)=\dim(\cH_{n-1})\le 6\Gamma_n.
\]
Moreover, by \cref{lem:gamma-comparisons},
\[
  \gamma(\rho^{-1}Q_n^+)\le \frac{15}{2}\Gamma_n.
\]
Applying \cref{thm:bernstein-predictable-dim} to $M^+$ yields
\[
  \P[\bigg]{\exists n\in\Np\colon
  \norm{(Q_n^+ + \rho I)^{-\frac12}M_n^+}
  >
  \frac12\sqrt{\rho}
  +
  \frac{39\Gamma_n+2y}{\sqrt{\rho}}}
  \le e^{-y}.
\]
Indeed,
\[
  2\roundb[\big]{\gamma(\rho^{-1}Q_n^+)+2D_n+y}
  \le
  2\roundb[\big]{\frac{15}{2}\Gamma_n+12\Gamma_n+y}
  =
  39\Gamma_n+2y.
\]

For the tail term, \cref{thm:hoeffding-fixed} applied to $M^-$ gives
\[
  \P[\bigg]{\exists n\in\Np\colon
  \norm{(V_n^-+\rho I)^{-\frac12}M_n^-}
  >
  \sqrt{2\roundb[\big]{\gamma(\rho^{-1}V_n^-)+y}}}
  \le e^{-y}.
\]
Since \cref{lem:online-truncation} gives $\gamma(\rho^{-1}V_n^-)\le 3\Gamma_n$, this implies
\[
  \P[\bigg]{\exists n\in\Np\colon
  \norm{(V_n^-+\rho I)^{-\frac12}M_n^-}
  >
  \sqrt{6\Gamma_n+2y}}
  \le e^{-y}.
\]
Combining the last two displays with the head--tail comparison and taking a union bound gives, with
probability at least $1-2e^{-y}$, for all $n\in\Np$,
\[
  \norm{(Q_n+\rho I)^{-\frac12}M_n}
  \le
  \sqrt{\rho}
  +
  \frac{78\Gamma_n+4y}{\sqrt{\rho}}
  +
  \sqrt{18\Gamma_n+6y}.
\]
This is the conclusion of \cref{thm:bernstein-fixed}. It remains only to prove the deterministic ingredients
used above.

\begin{lemma}[Online spectral truncation]\label{lem:online-truncation}
Fix $\rho\ge1$. There exists a predictable nested sequence $(\cH_n)_{n\in\N}$ of finite-dimensional
subspaces of $\cH$, with orthogonal projections $(\Pi_n)_{n\in\N}$, such that, defining
\[
  X_j^-=\Pi_{j-1}^{\perp}X_j
  \spaced{and}
  V_n^-=\sum_{j=1}^n X_j^-\otimes X_j^-,
\]
the following hold for all $n\in\Np$:
\[
  \norm{V_n^-}_{\op}\le \rho+1,\qquad
  \tr(V_n^-)\le 6\rho\,\gamma(\rho^{-1}V_n),\qquad
  \dim(\cH_n)\le 6\,\gamma(\rho^{-1}V_n),
\]
and
\[
  \gamma(\rho^{-1}V_n^-)\le 3\,\gamma(\rho^{-1}V_n).
\]
\end{lemma}

\begin{proof}
Set $\cH_0=\{0\}$. Having constructed $\cH_{n-1}$, let $\Pi_{n-1}$ be the orthogonal projection
onto $\cH_{n-1}$ and set
\[
  X_n^-=\Pi_{n-1}^{\perp}X_n
  \spaced{and}
  V_n^-=V_{n-1}^-+X_n^-\otimes X_n^-.
\]
Since $X_n$ and $\Pi_{n-1}$ are $\cF_{n-1}$-measurable, so is $V_n^-$. Let $\cE_n$ be the spectral
subspace of the restriction of $V_n^-$ to $\cH_{n-1}^{\perp}$ corresponding to eigenvalues at least $\rho$,
and define
\[
  \cH_n=\cH_{n-1}\oplus \cE_n.
\]
This gives a predictable nested sequence.

By induction, $V_n^-$ leaves both $\cH_n$ and $\cH_n^\perp$ invariant, and
\[
  \norm{\Pi_n^\perp V_n^-\Pi_n^\perp}_{\op}<\rho.
\]
Whenever a direction is added to $\cH_n$, its eigenvalue in $V_n^-$ is at least $\rho$ and at most $\rho+1$:
before the rank-one update, the residual operator on $\cH_{n-1}^{\perp}$ has norm strictly smaller than
$\rho$, and the update has operator norm at most one. Once a direction has been added, all later residuals
are orthogonal to it, so its eigenvalue in $V_t^-$ is frozen for $t\ge n$. Hence
\[
  \norm{V_n^-}_{\op}\le \rho+1
  \spaced{and}
  \rho\dim(\cH_n)\le \tr(V_n^-).
\]

It remains to control the residual trace. Define the ridge leverage scores
\[
  q_j=\langle X_j,(\rho I+V_{j-1})^{-1}X_j\rangle.
\]
Since $\rho\ge1$ and $\norm{X_j}\le1$, we have $q_j\le1$. The determinant update identity and the inequality
$x\le \log(1+x)/\log2$ for $x\in[0,1]$ give
\[
  \sum_{j=1}^n q_j
  \le
  \frac{1}{\log2}\sum_{j=1}^n\log(1+q_j)
  =
  \frac{2}{\log2}\gamma(\rho^{-1}V_n)
  \le
  3\gamma(\rho^{-1}V_n).
\]

We claim that $\norm{X_j^-}^2\le 2\rho q_j$. Let $R=\Pi_{j-1}^{\perp}$. Since $\cH_{j-1}$ is generated
by previously selected residual directions,
\[
  RV_{j-1}R=RV_{j-1}^-R\preceq \rho R.
\]
Using the Schur complement formula for $(\rho I+V_{j-1})^{-1}$ in the decomposition
$\cH_{j-1}\oplus\cH_{j-1}^{\perp}$,
\[
  q_j
  \ge
  \langle RX_j,(\rho R+RV_{j-1}R)^{-1}RX_j\rangle
  \ge
  \frac{\norm{X_j^-}^2}{2\rho}.
\]
Therefore,
\[
  \tr(V_n^-)
  =
  \sum_{j=1}^n\norm{X_j^-}^2
  \le
  2\rho\sum_{j=1}^n q_j
  \le
  6\rho\,\gamma(\rho^{-1}V_n).
\]
This proves the trace and dimension bounds.

Finally, since $\log(1+x)\le x$,
\[
  \gamma(\rho^{-1}V_n^-)
  =
  \frac12\sum_k\log(1+\lambda_k(V_n^-)/\rho)
  \le
  \frac{\tr(V_n^-)}{2\rho}
  \le
  3\,\gamma(\rho^{-1}V_n).
\]
\end{proof}

\begin{lemma}[Head--tail comparison]\label{lem:head-tail-comparison}
Fix $\rho\ge1$. Let $X_j=X_j^++X_j^-$ be a predictable decomposition and define
\[
  Q_n^+ =
  \sum_{j=1}^n \E[Y_j^2\mid\cF_{j-1}]X_j^+\otimes X_j^+,
  \qquad
  Q_n^- =
  \sum_{j=1}^n \E[Y_j^2\mid\cF_{j-1}]X_j^-\otimes X_j^-,
\]
and
\[
  M_n^+=\sum_{j=1}^nY_jX_j^+,\qquad
  M_n^-=\sum_{j=1}^nY_jX_j^-,
  \qquad
  V_n^-=\sum_{j=1}^nX_j^-\otimes X_j^-.
\]
Assume $\norm{V_n^-}_{\op}\le 2\rho$. Then
\[
  \norm{(Q_n+\rho I)^{-\frac12}(M_n^+ + M_n^-)}
  \le
  2\norm{(Q_n^++\rho I)^{-\frac12}M_n^+}
  +
  \sqrt{3}\norm{(V_n^-+\rho I)^{-\frac12}M_n^-}.
\]
\end{lemma}

\begin{proof}
For scalars $a,b\in\R$,
\[
  (a+b)^2+\frac12 b^2\ge \frac13 a^2.
\]
Applying this with $a=\langle z,X_j^+\rangle$ and $b=\langle z,X_j^-\rangle$, multiplying by
$\E[Y_j^2\mid\cF_{j-1}]$, and summing over $j$, gives
\[
  Q_n+\frac12 Q_n^-\succeq \frac13 Q_n^+.
\]
Since $Q_n^-\preceq V_n^-\preceq 2\rho I$, we have $\rho I\succeq Q_n^-/2$, and hence
\[
  Q_n+\rho I\succeq \frac13 Q_n^+.
\]
Also $Q_n+\rho I\succeq \rho I$. Taking the convex combination of these two lower bounds with weights
$3/4$ and $1/4$ yields
\[
  Q_n+\rho I\succeq \frac14(Q_n^++\rho I).
\]
Therefore,
\[
  \norm{(Q_n+\rho I)^{-\frac12}M_n^+}
  \le
  2\norm{(Q_n^++\rho I)^{-\frac12}M_n^+}.
\]

For the tail term, $Q_n+\rho I\succeq \rho I$, so
\[
  \norm{(Q_n+\rho I)^{-\frac12}M_n^-}
  \le
  \frac{\norm{M_n^-}}{\sqrt{\rho}}.
\]
Since $\norm{V_n^-}_{\op}\le 2\rho$,
\[
  \norm{M_n^-}
  \le
  \sqrt{\norm{V_n^-+\rho I}_{\op}}\,
  \norm{(V_n^-+\rho I)^{-\frac12}M_n^-}
  \le
  \sqrt{3\rho}\,
  \norm{(V_n^-+\rho I)^{-\frac12}M_n^-}.
\]
The result follows from the triangle inequality.
\end{proof}

\begin{lemma}[Information-gain comparison]\label{lem:gamma-comparisons}
With the notation of \cref{lem:online-truncation}, let
\[
  X_j^+=\Pi_{j-1}X_j,
  \qquad
  Q_n^+=\sum_{j=1}^n\E[Y_j^2\mid\cF_{j-1}]X_j^+\otimes X_j^+.
\]
Then
\[
  \gamma(\rho^{-1}Q_n^+)\le \frac{15}{2}\,\gamma(\rho^{-1}V_n).
\]
\end{lemma}

\begin{proof}
Since $\E[Y_j^2\mid\cF_{j-1}]\le1$, we have $Q_n^+\preceq V_n^+$, where
\[
  V_n^+=\sum_{j=1}^nX_j^+\otimes X_j^+.
\]
Also $X_j^+=X_j-X_j^-$, so for every $\epsilon>0$ and every $u\in\cH$,
\[
  \langle u,X_j^+\rangle^2
  \le
  (1+\epsilon)\langle u,X_j\rangle^2
  +(1+\epsilon^{-1})\langle u,X_j^-\rangle^2.
\]
Thus
\[
  V_n^+\preceq (1+\epsilon)V_n+(1+\epsilon^{-1})V_n^-.
\]
Using the determinant inequality
\[
  \det(I+A+B)\le \det(I+A)\det(I+B)
\]
for finite-rank positive operators $A,B$, together with $\gamma(aA)\le a\gamma(A)$ for $a\ge1$, gives
\[
  \gamma(\rho^{-1}Q_n^+)
  \le
  (1+\epsilon)\gamma(\rho^{-1}V_n)
  +(1+\epsilon^{-1})\gamma(\rho^{-1}V_n^-).
\]
Taking $\epsilon=\sqrt3$ and using \cref{lem:online-truncation},
\[
  \gamma(\rho^{-1}Q_n^+)
  \le
  \roundb[\big]{1+\sqrt3+3(1+1/\sqrt3)}\gamma(\rho^{-1}V_n)
  =
  (4+2\sqrt3)\gamma(\rho^{-1}V_n)
  \le
  \frac{15}{2}\gamma(\rho^{-1}V_n). \qedhere
\]
\end{proof}

\clearpage
\section{Confidence sequence for logistic regression (proof of \cref{thm:confidence-sequence})}
\label{appendix:confidence-sequence}

Observe first $\cL_n$ is $(2,1)$-generalised self-concordant in the sense of \citet{sun2019generalized}. Indeed,
\[
  \ell(u,y)=\log(1+e^u)-yu,
\]
so
\[
  \ell''(u,y)=\rV(u)
  \spaced{and}
  \ell'''(u,y)=\rV(u)(1-2\mu(u)).
\]
Thus $|\ell'''(u,y)|\le \ell''(u,y)$ for every $u\in\R$ and $y\in[0,1]$.
Consequently, for
\[
  \cL_n(f)=\sum_{i=1}^n \ell(\langle f,X_i\rangle,Y_i)+\frac{\rho}{2}\norm{f}^2,
\]
and for any $f,h,v\in\cH$,
\[
  D^3\cL_n(f)[h,v,v]
  =
  \sum_{i=1}^n
  \rV'(\langle f,X_i\rangle)
  \langle h,X_i\rangle
  \langle v,X_i\rangle^2.
\]
Using $\norm{X_i}\le1$ and the preceding one-dimensional bound,
\[
  \abs{D^3\cL_n(f)[h,v,v]}
  \le
  \norm{h}
  \sum_{i=1}^n
  \rV(\langle f,X_i\rangle)
  \langle v,X_i\rangle^2
  \le
  \norm{h}\,D^2\cL_n(f)[v,v],
\]
where the regularisation contributes no third derivative and only increases the second derivative. This proves that $\cL_n$ is $(2,1)$-generalised self-concordant, which will allow us to apply the following lemma,  extracted from \citet[Corollary~2 and Proposition~10]{sun2019generalized}. The results of \citet{sun2019generalized} are written for functions on $\Rd$, but the arguments carry over without modification to separable Hilbert spaces.

\begin{lemma}[Generalized self-concordance comparison]\label{lem:sc}
Let $F:\cH\to\R$ be a $(2,1)$-generalised self-concordant function. For $x,y\in\cH$, set
$t=\norm{y-x}$ and
\[
  G(x,y)=\int_0^1\nabla^2F(x+\tau(y-x))\dif\tau.
\]
Then
\[
  \upsilon(-t)\nabla^2F(x)
  \preceq
  G(x,y)
  \preceq
  \upsilon(t)\nabla^2F(x),
  \qquad
  \upsilon(t)=\frac{e^t-1}{t},
\]
with the convention $\upsilon(0)=1$.
\end{lemma}

\begin{proof}[Proof of \cref{thm:confidence-sequence}]
Fix $n\in\Np$ and write
\[
  H_n(f):=A_n(f)+\rho I,\qquad
  H^\star_n:=H_n(f^\star)=A^\star_n+\rho I,\qquad
  \hat H_n:=H_n(\hat f_n)=\hat A_n+\rho I.
\]
We also set
\[
  \Delta_n:=f^\star-\hat f_n
  \spaced{and}
  G_n:=\int_0^1 H_n(\hat f_n+\tau\Delta_n)\dif\tau.
\]

By \cref{lem:sc} and $\upsilon(-t)=(1-e^{-t})/t\ge 1/(1+t)$,
\begin{equation}\label{eq:logistic-hessian-endpoint-comparison}
  H^\star_n\preceq (1+\norm{\Delta_n})G_n
  \spaced{and}
  \hat H_n\preceq (1+\norm{\Delta_n})G_n.
\end{equation}
The first inequality follows by applying \cref{lem:sc} with the segment started at $f^\star$, and the second
with the segment started at $\hat f_n$.

We now relate estimation error to the score at $f^\star$. By the fundamental theorem of calculus for the gradient,
\[
  \nabla\cL_n(f^\star)-\nabla\cL_n(\hat f_n)=G_n\Delta_n.
\]
Since $\hat f_n$ minimises the $\rho$-regularised loss, $\nabla\cL_n(\hat f_n)=0$. Moreover,
\[
  \nabla\cL_n(f^\star)
  =
  \sum_{i=1}^n(\mu(\langle f^\star,X_i\rangle)-Y_i)X_i+\rho f^\star
  =
  \rho f^\star-M_n,
\]
where $M_n=\sum_{i=1}^n\epsilon_iX_i$ is the residual martingale from the main text. Define
\[
  Z_n:=\rho f^\star-M_n
  \spaced{and}
  q_n:=\norm{(H^\star_n)^{-\frac12}Z_n}.
\]
Thus
\begin{equation}\label{eq:score-identity}
  G_n\Delta_n=Z_n.
\end{equation}

The next few displays are deterministic. Put
\[
  a_n:=\norm{G_n^{\frac12}\Delta_n}.
\]
Using \eqref{eq:score-identity},
\[
  a_n^2
  =
  \langle \Delta_n,G_n\Delta_n\rangle
  =
  \langle Z_n,G_n^{-1}Z_n\rangle.
\]
By \eqref{eq:logistic-hessian-endpoint-comparison}, $G_n^{-1}\preceq(1+\norm{\Delta_n})(H^\star_n)^{-1}$.
Since $G_n\succeq\rho I$, we also have $\norm{\Delta_n}\le a_n/\sqrt{\rho}$. Therefore
\[
  a_n^2
  \le
  \roundb[\bigg]{1+\frac{a_n}{\sqrt{\rho}}}q_n^2.
\]
Equivalently,
\[
  a_n^2\le a_n\frac{q_n^2}{\sqrt{\rho}}+q_n^2,
\]
and the elementary implication $x^2\le xb+c\Rightarrow x\le b+\sqrt c$ gives
\begin{equation}\label{eq:an-bound}
  a_n\le q_n\roundb[\bigg]{1+\frac{q_n}{\sqrt{\rho}}}.
\end{equation}

Next let $K_n$ denote either $H^\star_n$ or $\hat H_n$, and set
\[
  w_n(K_n):=\norm{K_n^{\frac12}\Delta_n}.
\]
By \eqref{eq:logistic-hessian-endpoint-comparison},
\[
  w_n(K_n)^2
  \le
  (1+\norm{\Delta_n})a_n^2.
\]
Since $K_n\succeq\rho I$, $\norm{\Delta_n}\le w_n(K_n)/\sqrt{\rho}$, and hence
\[
  w_n(K_n)^2
  \le
  w_n(K_n)\frac{a_n^2}{\sqrt{\rho}}+a_n^2.
\]
Applying the same elementary implication again yields
\begin{equation}\label{eq:wn-bound}
  w_n(K_n)\le a_n+\frac{a_n^2}{\sqrt{\rho}}.
\end{equation}
Let
\[
  r_n:=\frac{q_n}{\sqrt{\rho}}.
\]
Combining \eqref{eq:an-bound} and \eqref{eq:wn-bound} gives
\[
  w_n(K_n)
  \le
  q_n(1+r_n)\roundb[\big]{1+r_n(1+r_n)}
  =
  q_n(1+2r_n+2r_n^2+r_n^3).
\]
Since $1+2r+2r^2+r^3\le 6+2r^3$ for every $r\ge0$, we obtain
\begin{equation}\label{eq:hessian-error-deterministic}
  \norm{(H^\star_n)^{\frac12}\Delta_n}
  \vee
  \norm{\hat H_n^{\frac12}\Delta_n}
  \le
  q_n(6+2r_n^3).
\end{equation}
Also, from \eqref{eq:an-bound} and $G_n\succeq\rho I$,
\begin{equation}\label{eq:ambient-error-deterministic}
  \norm{\Delta_n}
  \le
  \frac{a_n}{\sqrt{\rho}}
  \le
  r_n(1+r_n).
\end{equation}

It remains to control $q_n$. We first compare $Q_n$ and $A^\star_n$. Since $0\le Y_i\le1$ and
$p_i:=\mu(\langle f^\star,X_i\rangle)$ is the conditional mean of $Y_i$,
\[
  \E[\epsilon_i^2\mid\cF_{i-1}]
  =
  \E[Y_i^2\mid\cF_{i-1}]-p_i^2
  \le
  p_i-p_i^2
  =
  \rV(\langle f^\star,X_i\rangle).
\]
Thus $Q_n\preceq A^\star_n$, and so $Q_n+\rho I\preceq H^\star_n$. Since inversion reverses the
positive-semidefinite order,
\[
  \norm{(H^\star_n)^{-\frac12}M_n}
  \le
  \norm{(Q_n+\rho I)^{-\frac12}M_n}.
\]
Therefore, by the triangle inequality and $\norm{f^\star}\le b$,
\[
  q_n
  =
  \norm{(H^\star_n)^{-\frac12}(\rho f^\star-M_n)}
  \le
  \norm{(Q_n+\rho I)^{-\frac12}M_n}+\sqrt{\rho}\,b.
\]
On the event $\cC$ from the main text, this gives, simultaneously for all $n\in\Np$,
\[
  r_n=\frac{q_n}{\sqrt{\rho}}
  \le
  \frac{\beta_n(\rho,y)}{\sqrt{\rho}}+b
  =
  u_n.
\]
Substituting this into \eqref{eq:hessian-error-deterministic} and \eqref{eq:ambient-error-deterministic} yields,
on $\cC$ and uniformly over $n\in\Np$,
\[
  \norm{(A^\star_n+\rho I)^{\frac12}(f^\star-\hat f_n)}
  \vee
  \norm{(\hat A_n+\rho I)^{\frac12}(f^\star-\hat f_n)}
  \le
  \sqrt{\rho}\,u_n(6+2u_n^3)
  =
  \omega_n,
\]
and
\[
  \norm{f^\star-\hat f_n}
  \le
  u_n(1+u_n)
  =
  \alpha_n.
\]
This proves \cref{thm:confidence-sequence}.
\end{proof}

\clearpage
\section{Regret bound for logistic bandits (proof of \cref{thm:raw_reg})}
\label{appendix:regret}

\regretThm*

Work on the event where the confidence sequence of \cref{thm:confidence-sequence} holds for all times.
Recall that
\[
  A^\star_n=A_n(f^\star),
  \qquad
  \hat A_n=A_n(\hat f_n),
\]
and define
\[
  H^\star_n:=A^\star_n+\rho I,
  \qquad
  \bar H_n:=\hat A_n+(\rho+\omega_n^2)I,
  \qquad
  \Gamma^\star_n:=\gamma(\rho^{-1}A^\star_n).
\]
On the confidence event, for every $n\in\Np$,
\[
  \norm{(H^\star_n)^{\frac12}(f^\star-\hat f_n)}
  \vee
  \norm{(\hat A_n+\rho I)^{\frac12}(f^\star-\hat f_n)}
  \le \omega_n,
  \qquad
  \norm{f^\star-\hat f_n}\le \alpha_n.
\]
It follows that $f^\star\in\bar{\cE}_n$ for every $n\ge1$, since
\[
  \norm{\bar H_n^{\frac12}(f^\star-\hat f_n)}^2
  =
  \norm{(\hat A_n+\rho I)^{\frac12}(f^\star-\hat f_n)}^2
  +\omega_n^2\norm{f^\star-\hat f_n}^2
  \le
  \omega_n^2(1+\alpha_n^2)
  \le
  \bar\omega_n^2.
\]
Also $f^\star\in\bar{\cE}_0$ by assumption.

We shall use the following curvature comparison.

\begin{lemma}[Curvature comparison]\label{lem:curvature-comparison}
For every $n\in\Np$,
\begin{equation*}
  H^\star_n \preceq 2\bar H_n .
\end{equation*}
\end{lemma}
\begin{proof}
    Fix $z\in\cH$ and write
\[
  \Delta=f^\star-\hat f_n,\qquad
  a_i=\langle \Delta,X_i\rangle,\qquad
  b_i=\langle z,X_i\rangle,
\]
and
\[
  v_i^\star=\rV(\langle f^\star,X_i\rangle),
  \qquad
  \hat v_i=\rV(\langle \hat f_n,X_i\rangle).
\]
The logistic variance satisfies
\[
  \abs[\big]{(\log \rV)'(u)}=\abs[\big]{1-2\mu(u)}\le1.
\]
Hence, for any $r,s\in\R$,
\[
  \abs{\rV(r)-\rV(s)}
  \le
  \frac{|r-s|}{2}\{\rV(r)+\rV(s)\}.
\]
Indeed, applying the preceding log-Lipschitz bound to $x=\rV(r)$ and $y=\rV(s)$ gives
$\abs{\log x-\log y}\le |r-s|$, and therefore
\[
  \frac{|x-y|}{x+y}
  \le \tanh\!\left(\frac{|r-s|}{2}\right)
  \le \frac{|r-s|}{2}.
\]
Let
\[
  s_z=\sum_{i=1}^n v_i^\star b_i^2,
  \qquad
  h_z=\sum_{i=1}^n \hat v_i b_i^2.
\]
Using the scalar comparison above and Cauchy--Schwarz,
\[
\begin{aligned}
  s_z
  \le
  h_z+\frac12\sum_{i=1}^n |a_i|(v_i^\star+\hat v_i)b_i^2
  \le
  h_z+\frac12
  \left(\sum_{i=1}^n(v_i^\star+\hat v_i)a_i^2\right)^{1/2}
  \left(\sum_{i=1}^n(v_i^\star+\hat v_i)b_i^4\right)^{1/2}.
\end{aligned}
\]
By the two confidence-sequence inequalities,
\[
  \sum_{i=1}^n(v_i^\star+\hat v_i)a_i^2\le 2\omega_n^2,
\]
and since $\norm{X_i}\le1$,
\[
  \sum_{i=1}^n(v_i^\star+\hat v_i)b_i^4
  \le
  \norm{z}^2(s_z+h_z).
\]
Thus
\[
  s_z\le h_z+\frac{\omega_n}{\sqrt2}\norm{z}\sqrt{s_z+h_z}.
\]
Put $c=\omega_n\norm{z}/\sqrt2$ and $w=\sqrt{s_z+h_z}$. Then
$w^2\le 2h_z+cw$, and hence $w\le c+\sqrt{2h_z}$. Consequently,
\[
  s_z
  \le
  h_z+c(c+\sqrt{2h_z})
  \le
  2h_z+\frac32c^2
  =
  2h_z+\frac34\omega_n^2\norm{z}^2.
\]
Adding $\rho\norm{z}^2$ to both sides gives
\[
  \langle z,H^\star_n z\rangle
  \le
  2\langle z,\hat A_n z\rangle
  +(\rho+\tfrac34\omega_n^2)\norm{z}^2
  \le
  2\langle z,\bar H_n z\rangle. \qedhere
\]
\end{proof}

\begin{proof}[Proof of \cref{thm:raw_reg}]
For each round $j$, define
\[
  u^\star=\langle f^\star,x^\star\rangle,\qquad
  u^\star_j=\langle f^\star,X_j\rangle,\qquad
  r_j=\mu(u^\star)-\mu(u^\star_j),
  \qquad
  g_j=u^\star-u^\star_j.
\]
Since $\mu$ is increasing, $r_j\ge0$ and $g_j\ge0$.

We first control the linear gap $g_j$. For $j\ge2$, since $f^\star\in\bar{\cE}_{j-1}$ and $X_j$ is chosen
optimistically,
\[
  u^\star
  \le
  \sup_{f\in\bar{\cE}_{j-1}}\langle f,X_j\rangle.
\]
On the other hand, again because $f^\star\in\bar{\cE}_{j-1}$,
\[
  \langle f^\star,X_j\rangle
  \ge
  \langle \hat f_{j-1},X_j\rangle
  -
  \bar\omega_{j-1}\norm{\bar H_{j-1}^{-\frac12}X_j}.
\]
The supremum over $\bar{\cE}_{j-1}$ is
\[
  \langle \hat f_{j-1},X_j\rangle
  +
  \bar\omega_{j-1}\norm{\bar H_{j-1}^{-\frac12}X_j}.
\]
Therefore,
\[
  g_j
  \le
  2\bar\omega_{j-1}\norm{\bar H_{j-1}^{-\frac12}X_j},
  \qquad j\ge2.
\]
For $j=1$, the same bound below follows from the crude estimate
$g_1\le2b\norm{X_1}$ and the fact that $\bar\omega_n\ge \sqrt{\rho}\,b$.

Let
\[
  q_j
  :=
  \rV(u^\star_j)\norm{(H^\star_{j-1})^{-\frac12}X_j}^2,
  \qquad j\in\Np,
\]
where $H^\star_0=\rho I$. By \cref{lem:curvature-comparison},
$\bar H_{j-1}^{-1}\preceq 2(H^\star_{j-1})^{-1}$ for $j\ge2$. Since $\bar\omega_j$ is nondecreasing on the
event under consideration, we get, for all $n\in\Np$,
\[
  \sum_{j=1}^n \rV(u^\star_j)g_j^2
  \le
  8\bar\omega_n^2\sum_{j=1}^n q_j.
\]
The usual elliptical-potential argument applied to
\[
  A^\star_j=A^\star_{j-1}+\rV(u^\star_j)X_j\otimes X_j
\]
gives
\[
  \sum_{j=1}^n q_j
  \le
  2\sum_{j=1}^n\log(1+q_j)
  =
  2\log\det(I+\rho^{-1}A^\star_n)
  =
  4\Gamma^\star_n,
\]
where we used $q_j\le1$. Hence, with
\[
  B_n:=\sum_{j=1}^n \rV(u^\star_j)g_j^2,
\]
we have established
\begin{equation}\label{eq:Bn-bound-logistic}
  B_n\le 32\bar\omega_n^2\Gamma^\star_n.
\end{equation}

We now convert this width bound into regret. Taylor's theorem gives, for every $j$,
\[
  r_j
  =
  \mu(u^\star)-\mu(u^\star_j)
  \le
  \rV(u^\star_j)g_j+\frac18g_j^2,
\]
because $\mu''(u)\le1/4$ for all $u\in\R$. Summing the first term and applying Cauchy--Schwarz gives
\[
  \sum_{j=1}^n \rV(u^\star_j)g_j
  \le
  \left(\sum_{j=1}^n \rV(u^\star_j)\right)^{1/2}B_n^{1/2}.
\]
Let $p_j=\mu(u^\star_j)$ and $p^\star=\mu(u^\star)$. Since $p_j\le p^\star$ and
$\abs{p_j(1-p_j)-p^\star(1-p^\star)}\le p^\star-p_j$, we have
\[
  \rV(u^\star_j)\le \nu_\star+r_j.
\]
Therefore,
\[
  \sum_{j=1}^n \rV(u^\star_j)\le n\nu_\star+R_n.
\]

For the second-order Taylor term, note that $|u^\star_j|\le b$, and hence
\[
  \frac{1}{\rV(u^\star_j)}
  =
  2+e^{u^\star_j}+e^{-u^\star_j}
  \le 4e^b.
\]
Thus
\[
  \sum_{j=1}^n g_j^2
  \le
  4e^bB_n.
\]
Combining the last displays,
\[
  R_n
  \le
  \sqrt{(n\nu_\star+R_n)B_n}
  +\frac12e^bB_n
  \le
  \sqrt{n\nu_\star B_n}
  +\sqrt{R_nB_n}
  +\frac12e^bB_n.
\]
Using $\sqrt{R_nB_n}\le R_n/2+B_n/2$ and $e^b\ge1$, we obtain
\[
  R_n
  \le
  2\sqrt{n\nu_\star B_n}
  +2e^bB_n.
\]
Finally, substituting \eqref{eq:Bn-bound-logistic} yields
\[
  R_n
  \le
  2\sqrt{32}\,\bar\omega_n\sqrt{\nu_\star n\Gamma^\star_n}
  +
  64e^b\bar\omega_n^2\Gamma^\star_n
  \le
  12\bar\omega_n\sqrt{\nu_\star n\Gamma^\star_n}
  +
  64e^b\bar\omega_n^2\Gamma^\star_n.
\]
This holds simultaneously for all $n\in\Np$ on the confidence event, whose probability is at least
$1-Ce^{-y}$ by assumption.
\end{proof}

%% file: main.bbl
\begin{thebibliography}{27}
\providecommand{\natexlab}[1]{#1}
\providecommand{\url}[1]{\texttt{#1}}
\expandafter\ifx\csname urlstyle\endcsname\relax
  \providecommand{\doi}[1]{doi: #1}\else
  \providecommand{\doi}{doi: \begingroup \urlstyle{rm}\Url}\fi

\bibitem[Abbasi-Yadkori(2013)]{abbasi2013online}
Yasin Abbasi-Yadkori.
\newblock \emph{Online Learning for Linearly Parametrized Control Problems}.
\newblock PhD thesis, University of Alberta, 2013.

\bibitem[Abbasi-Yadkori et~al.(2011)Abbasi-Yadkori, P{\'a}l, and
  Szepesv{\'a}ri]{abbasi2011improved}
Yasin Abbasi-Yadkori, D{\'a}vid P{\'a}l, and Csaba Szepesv{\'a}ri.
\newblock Improved algorithms for linear stochastic bandits.
\newblock In \emph{Advances in Neural Information Processing Systems}, 2011.

\bibitem[Bakhtiari et~al.(2025)Bakhtiari, Ayoub, Robertson, Janz, and
  Szepesv{\'a}ri]{bakhtiari2026eluder}
Alireza Bakhtiari, Alex Ayoub, Samuel Robertson, David Janz, and Csaba
  Szepesv{\'a}ri.
\newblock Eluder dimension: localise it!
\newblock In \emph{Advances in Neural Information Processing Systems}, 2025.

\bibitem[Chugg \& Ramdas(2025)Chugg and Ramdas]{chugg2025variational}
Ben Chugg and Aaditya Ramdas.
\newblock A variational approach to dimension-free self-normalized
  concentration.
\newblock \emph{arXiv preprint arXiv:2508.06483}, 2025.

\bibitem[de~la Pe{\~n}a(1999)]{de1999general}
Victor~H. de~la Pe{\~n}a.
\newblock A general class of exponential inequalities for martingales and
  ratios.
\newblock \emph{The Annals of Probability}, 27\penalty0 (1):\penalty0 537--564,
  1999.

\bibitem[de~la Pe{\~n}a et~al.(2004)de~la Pe{\~n}a, Klass, and Lai]{de2004self}
Victor~H. de~la Pe{\~n}a, Michael~J. Klass, and Tze~Leung Lai.
\newblock Self-normalized processes: Exponential inequalities, moment bounds
  and iterated logarithm laws.
\newblock \emph{The Annals of Probability}, 32\penalty0 (3):\penalty0
  1902--1933, 2004.

\bibitem[de~la Pe{\~n}a et~al.(2009)de~la Pe{\~n}a, Lai, and
  Shao]{pena2009self}
Victor~H. de~la Pe{\~n}a, Tze~Leung Lai, and Qi-Man Shao.
\newblock \emph{Self-Normalized Processes: Limit Theory and Statistical
  Applications}.
\newblock Springer, 2009.

\bibitem[Efron(1969)]{efron1969student}
Bradley Efron.
\newblock Student's {$t$}-test under symmetry conditions.
\newblock \emph{Journal of the American Statistical Association}, 64\penalty0
  (328):\penalty0 1278--1302, 1969.

\bibitem[Faury et~al.(2020)Faury, Abeille, Calauz{\`e}nes, and
  Fercoq]{faury2020improved}
Louis Faury, Marc Abeille, Cl{\'e}ment Calauz{\`e}nes, and Olivier Fercoq.
\newblock Improved optimistic algorithms for logistic bandits.
\newblock In \emph{International Conference on Machine Learning}, 2020.

\bibitem[Ledoux \& Talagrand(1991)Ledoux and Talagrand]{ledoux1991probability}
Michel Ledoux and Michel Talagrand.
\newblock \emph{Probability in {Banach} Spaces: Isoperimetry and Processes},
  volume~23.
\newblock Springer, 1991.

\bibitem[Lee et~al.(2024)Lee, Yun, and Jun]{lee2024unified}
Jungyhun Lee, Se-Young Yun, and Kwang-Sung Jun.
\newblock A unified confidence sequence for generalized linear models, with
  applications to bandits.
\newblock In \emph{Advances in Neural Information Processing Systems}, 2024.

\bibitem[Li et~al.(2010)Li, Chu, Langford, and Schapire]{li2010contextual}
Lihong Li, Wei Chu, John Langford, and Robert~E. Schapire.
\newblock A contextual-bandit approach to personalized news article
  recommendation.
\newblock In \emph{Proceedings of the 19th International Conference on World
  Wide Web}, 2010.

\bibitem[Li et~al.(2017)Li, Lu, and Zhou]{li2017provably}
Lihong Li, Yu~Lu, and Dengyong Zhou.
\newblock Provably optimal algorithms for generalized linear contextual
  bandits.
\newblock In \emph{International Conference on Machine Learning}, 2017.

\bibitem[Li \& Scarlett(2022)Li and Scarlett]{li2022gaussian}
Zihan Li and Jonathan Scarlett.
\newblock {Gaussian} process bandit optimization with few batches.
\newblock In \emph{International Conference on Artificial Intelligence and
  Statistics}, 2022.

\bibitem[Logan et~al.(1973)Logan, Mallows, Rice, and Shepp]{logan1973limit}
Benjamin~F. Logan, Colin~L. Mallows, Stephen~O. Rice, and Lawrence~A. Shepp.
\newblock Limit distributions of self-normalized sums.
\newblock \emph{The Annals of Probability}, 1\penalty0 (5):\penalty0 788--809,
  1973.

\bibitem[Luo(2022)]{luo2022azuma}
Sijie Luo.
\newblock On {Azuma}-type inequalities for {Banach} space-valued martingales.
\newblock \emph{Journal of Theoretical Probability}, 35\penalty0 (2):\penalty0
  772--800, 2022.

\bibitem[Martinez-Taboada \& Ramdas(2026)Martinez-Taboada and
  Ramdas]{martinez2024empirical}
Diego Martinez-Taboada and Aaditya Ramdas.
\newblock Empirical {Bernstein} in smooth {Banach} spaces.
\newblock \emph{Annals of Applied Probability}, 2026.
\newblock Forthcoming.

\bibitem[Metelli et~al.(2025)Metelli, Drago, and Mussi]{metelli2025generalized}
Alberto~Maria Metelli, Simone Drago, and Marco Mussi.
\newblock Generalized kernelized bandits: A novel self-normalized
  {Bernstein}-like dimension-free inequality and regret bounds.
\newblock \emph{arXiv preprint arXiv:2508.01681}, 2025.

\bibitem[Pinelis(1994)]{pinelis1994optimum}
Iosif Pinelis.
\newblock Optimum bounds for the distributions of martingales in {Banach}
  spaces.
\newblock \emph{The Annals of Probability}, 22\penalty0 (4):\penalty0
  1679--1706, 1994.

\bibitem[Seeger et~al.(2008)Seeger, Kakade, and Foster]{seeger2008information}
Matthias~W. Seeger, Sham~M. Kakade, and Dean~P. Foster.
\newblock Information consistency of nonparametric {Gaussian} process methods.
\newblock \emph{IEEE Transactions on Information Theory}, 54\penalty0
  (5):\penalty0 2376--2382, 2008.

\bibitem[Srinivas et~al.(2010)Srinivas, Krause, Kakade, and
  Seeger]{srinivas2009gaussian}
Niranjan Srinivas, Andreas Krause, Sham~M. Kakade, and Matthias Seeger.
\newblock {Gaussian} process optimization in the bandit setting: No regret and
  experimental design.
\newblock In \emph{International Conference on Machine Learning}, 2010.

\bibitem[Sun \& Tran-Dinh(2019)Sun and Tran-Dinh]{sun2019generalized}
Tianxiao Sun and Quoc Tran-Dinh.
\newblock Generalized self-concordant functions: A recipe for {Newton}-type
  methods.
\newblock \emph{Mathematical Programming}, 178\penalty0 (1--2):\penalty0
  145--213, 2019.

\bibitem[Valko et~al.(2013)Valko, Korda, Munos, Flaounas, and
  Cristianini]{valko2013finite}
Michal Valko, Nathan Korda, R{\'e}mi Munos, Ilias Flaounas, and Nello
  Cristianini.
\newblock Finite-time analysis of kernelised contextual bandits.
\newblock In \emph{Uncertainty in Artificial Intelligence}, 2013.

\bibitem[Zelen(1969)]{zelen1969play}
Marvin Zelen.
\newblock Play the winner rule and the controlled clinical trial.
\newblock \emph{Journal of the American Statistical Association}, 64\penalty0
  (325):\penalty0 131--146, 1969.

\bibitem[Zhang(2002)]{zhang2002effective}
Tong Zhang.
\newblock Effective dimension and generalization of kernel learning.
\newblock In \emph{Advances in Neural Information Processing Systems}, 2002.

\bibitem[Zhang(2005)]{zhang2005learning}
Tong Zhang.
\newblock Learning bounds for kernel regression using effective data
  dimensionality.
\newblock \emph{Neural Computation}, 17\penalty0 (9):\penalty0 2077--2098,
  2005.

\bibitem[Ziemann(2025)]{ziemann2024vector}
Ingvar Ziemann.
\newblock A vector {Bernstein} inequality for self-normalized martingales.
\newblock \emph{Transactions on Machine Learning Research}, 2025.

\end{thebibliography}
